\documentclass[11pt]{amsart}
  \usepackage{amsmath,amsthm,amsfonts,amssymb}

  \newtheorem{thm}{Theorem}[section]
  \newtheorem{lem}[thm]{Lemma}
  \newtheorem{que}[thm]{Question}
  \newtheorem{cor}[thm]{Corollary}
  
  \newtheorem*{thmA}{Theorem A}
  \newtheorem*{thmB}{Theorem B}
  \newtheorem*{thmC}{Theorem  C}
    \newtheorem*{thmE}{Theorem E}
  \newtheorem*{thmD}{Theorem  D}
    \newtheorem*{thmF}{Theorem  F}
  \newtheorem*{corG}{Corollary G}
  \newtheorem*{thm1.1}{Theorem}
  \newtheorem*{thm1.2}{Theorem}

\newcommand{\Aut}{{\operatorname{Aut}}}
\newcommand{\rdim}{{\operatorname{rdim}}}
\newcommand{\GL}{{\operatorname{GL}}}
\newcommand{\SL}{{\operatorname{SL}}}
\newcommand{\Soc}{{\operatorname{Soc}}}
\newcommand{\core}{{\operatorname{core}}}

\newcommand{\Irr}{{\operatorname{Irr}}}
\newcommand{\Ker}{{\operatorname{Ker}}}

\newcommand{\GF}{{\operatorname{GF}}}
\newcommand{\ed}{{\operatorname{ed}}}

  \title[Minimal dimension of a faithful representation]{On the minimal dimension of a faithful linear representation of a finite group}
  \author{}
  \date{}

\begin{document}

  \maketitle

 \bigskip
  \centerline{by}
  \bigskip

 \smallskip
  \centerline{Alexander Moret\'o}  \centerline{Departament
  de Matem\`atiques} \centerline{Universitat de Val\`encia}
  \centerline{46100 Burjassot. Val\`encia SPAIN} \centerline{  Alexander.Moreto@uv.es}

 \vskip 10pt

{\bf Abstract.}  The representation dimension of a finite group $G$ is  the minimal dimension of a faithful complex linear representation of  $G$. We prove that the representation dimension of any finite group $G$  is at most $\sqrt{|G|}$ except if $G$ is a $2$-group with elementary abelian center of order $8$ and all irreducible characters of $G$ whose kernel does not contain $Z(G)$ are fully ramified with respect to $G/Z(G)$. We also obtain bounds for the representation dimension of quotients of $G$ in terms of the representation dimension of $G$, and discuss the relation of this invariant with the essential dimension of $G$. 


  \vfill

  \noindent   Research  supported by Ministerio de Ciencia e Innovaci\'on (Grant PID2019-103854GB-I00 funded by MCIN/AEI/ 10.13039/501100011033 )   and Generalitat Valenciana AICO/2020/298. I thank D. Holt, Z. Reichstein and G. Robinson for helpful comments.  In particular, Z. Reichstein asked the question that led to Theorem D. I am also indebted to D. Rae for pointing out a mistake in a previous version of this paper. 
  
 \section{Introduction}
 
 Given a positive integer $n$, the study of the (finite complex) linear groups of degree $n$ has been a classical theme of research in finite group theory.  For instance, in 1878, C. Jordan proved that if $G$ is a linear group of degree $n$, then there exists $A\trianglelefteq G$ abelian such that $|G:A|\leq j(n)$ for some integer valued function $j(n)$ (see \cite{rob} for a modern classification-free proof of this theorem and for a description of earlier proofs.) After the classification of finite simple groups was completed, sharp bounds for the function $j(n)$ were found by M. Collins \cite{col} in 2008, improving on an earlier unpublished manuscript by B. Weisfeiler. 
  
 Following \cite{ckr},  let $\rdim(G)$ be the minimal integer such that a finite group $G$ embeds into $\GL(\rdim(G),\mathbb{C})$, i.e.,   $\rdim(G)$ is the smallest integer $n$ such that a finite group $G$ is a linear group of degree $n$. This was called the representation dimension in \cite{ckr}. Clearly, $\rdim(G)\leq|G|$. Surprisingly, this natural invariant of a finite group has not been very studied from a group-theoretic point of view. Recently, it has been proven to be very relevant in a large number of areas outside finite group representation theory. See for instance the Preface of \cite{tot} for its relevance in group cohomology theory or \cite{bg, gow} for its relevance to show that certain Cayley graphs are expander graphs. All the nontrivial results we are aware of  on $\rdim(G)$ when $G$ is not close to a simple group have been motivated  by the so-called essential dimension of a finite group $\ed(G)$.  
  This concept was introduced in 1997  by J. Buhler and Z. Reichstein in \cite{br} with motivations from algebraic geometry. Since then it has found applications in a large number of areas (see \cite{mer1, mer2}).  It is known that $\ed(G)\leq \rdim(G)$ (see Proposition 4.15 of \cite{bf}). Both $\ed(G)$ and $\rdim(G)$ depend on the field  over which we are considering the representations of $G$ and are of interest over arbitrary fields. For simplicity, in this note we will restrict ourselves to the field of complex numbers, although our arguments work over any field with sufficiently many roots of unity. A major result was the proof by N. Karpenko and A. Merkurjev \cite{km} that $\ed(G)=\rdim(G)$ when $G$ is a $p$-group. This has motivated the study of $\rdim(G)$ for several families of $p$-groups. See \cite{ mr, ckr, bms1, bms2}.
  
  In this paper, prompted by a question raised on the Math Overflow web site, we study the problem of finding sharp bounds for $\rdim(G)$ in terms of $|G|$. More precisely, the question asked was  whether $\rdim(G)\leq\sqrt{|G|}$. As pointed out by D. Holt, $C_2\times C_2\times C_2$ is a counterexample.  Our first main result shows that all counterexamples are closely related to Holt's example. In the following statement, $\Soc(G)$ is the socle of $G$ and $\Irr(G|Z(G))$ is the set of irreducible characters of $G$ that lie over a nonprincipal linear character of the center of $G$. We refer the reader to Problem 6.3 of \cite{isa} for the definition of fully ramified character.

  \begin{thmA}
  Let $G$ be a finite  group. Then one of the following holds:
  \begin{enumerate}
  \item
  $\rdim(G)\leq\sqrt{|G|}$; 
  \item
   $G$ is a $2$-group with socle $\Soc(G)=Z(G)=C_2\times C_2\times C_2$ and all characters in $\Irr(G|Z(G))$ are fully ramified with respect to $G/Z(G)$. For any such group $G$, 
   $\rdim(G)=\frac{3}{\sqrt{8}}\sqrt{|G|}$. 
   \end{enumerate}
  \end{thmA}
  
  We have also shown that the equality $\rdim(G)=\sqrt{|G|}$ just holds in groups that are similar to those in (ii) above.
  
  \begin{thmB}
  Let $G$ be a finite group. Then  $\rdim(G)=\sqrt{|G|}$  if and only if one of the following holds:
  \begin{enumerate}
  \item
   $G$ is a $2$-group with socle $\Soc(G)=Z(G)=C_2\times C_2$ and all characters in $\Irr(G|Z(G))$ are fully ramified with respect to $Z(G)$.
   \item
    $G$ is a $2$-group with socle $\Soc(G)=Z(G)=C_2\times C_2\times C_2\times C_2$ and all characters in $\Irr(G|Z(G))$ are fully ramified with respect to $Z(G)$.
    \end{enumerate}
  \end{thmB}
  
  It is interesting to note that the groups with $\rdim(G)$ as large as it can be in comparison with $|G|$ turn out to be $2$-groups. There are nonabelian groups among those that appear in Theorem A (ii) and in Theorem B: consider for instance the Sylow $2$-subgroups of $\GL(3,8)$ and $\GL(3,4)$ and $\GL(3,16)$. They are examples of nonabelian groups with the structure described in Theorem A(ii), Theorem B(i) and Theorem B(ii), respectively.  They are also examples of the so-called Heisenberg groups. 
  There are nonnilpotent groups $G$ with $\rdim(G)$ arbitrarily close to $\sqrt{|G|}$: consider the Frobenius groups of order $(p-1)p$ for any prime $p$. 
  
  For any positive integer $n$ and prime $p$, the maximal representation dimension among $p$-groups of order $p^n$ was determined in \cite{ckr}. As pointed out in that paper, if $f_p(n)=\max_{r\in\mathbb{N}}(rp^{\lfloor(n-r)/2\rfloor})$, then $\rdim(G)\leq f_p(n)$ for any $G$ $p$-group of order $p^n$. It was shown that, with a few exceptions for $p$ and $n$ listed in Theorem 1 of \cite{ckr}, there exists a $p$-group $G$ of order $p^n$ such that $\rdim(G)=f_p(n)$. When $p=2$ the exceptional cases are $n=5$ and $n=7$.  In the next result we characterize the $2$-groups where this equality occurs. It will be used in the proof of Theorems A and B.
    
  \begin{thmC}
  Let $G$ be a $2$-group of order $2^n$ for some positive integer $n\not\in\{1,5,7\}$. Then $\rdim(G)=f_2(n)$ if and only if one of the following holds:
  \begin{enumerate}
  \item 
  $n$ is odd, $Z(G)$ is elementary abelian of order $8$ and all characters in $\Irr(G|Z(G))$ are fully ramified with respect to $G/Z(G)$.
  \item
  $n$ is even,  $Z(G)$ is elementary abelian of order either $4$ or $16$ and all characters in $\Irr(G|Z(G))$ are fully ramified with respect to $G/Z(G)$.
  \end{enumerate}
  \end{thmC}
  
  When $p$ is odd, the exceptional case in Theorem 1 of \cite{ckr} is $n=4$. We have the following.
  
  \begin{thmD}
  Let $p$ be an odd prime and let $G$ be a $p$-group of order $p^n$ for some positive integer $n\not\in\{1,4\}$. Then $\rdim(G)=f_p(n)$ if and only if one of the following holds:
  \begin{enumerate}
  \item 
  $n$ is odd, $Z(G)$ has order $p$ and all characters in $\Irr(G|Z(G))$ are fully ramified with respect to $G/Z(G)$.
  \item
  $n$ is even,  $Z(G)$ is elementary abelian of order $p^2$  and all characters in $\Irr(G|Z(G))$ are fully ramified with respect to $G/Z(G)$.
  \end{enumerate}
  \end{thmD}
  
  We obtain Theorems C and D as a consequence of a more general characterization of $p$-groups with center of rank $r$ and $\rdim(G)=rp^{\lfloor(n-r)/2\rfloor}$ (see Theorem \ref{even} and Theorem \ref{odd}).

  There is a related invariant that has been more studied with a group-theoretic motivation. This is the smallest dimension of a faithful permutation representation, denoted by $\mu(G)$.  It was shown by P. Neumann \cite{neu} that there are groups $G$ with normal subgroups $N$ such that $\mu(G/N)>\mu(N)$. 
  L.Kov\'acs and C. Praeger \cite{kp} showed that $\mu(G/N)\leq\mu(G)$ holds whenever $G/N$ does not have nontrivial abelian normal subgroups. Holt and J. Walton \cite{hw} proved that $\mu(G/N)\leq(4.5)^{\mu(G)-1}$.  As, for instance,  the double cover of $M_{12}$ shows, it is not true that $\rdim(G/N)\leq \rdim(G)$ even when $G/N$ does not have nontrivial abelian normal subgroups. As a consequence of Jordan's theorem, it is easy to obtain the following variation of the Holt-Walton theorem for representation dimension.
  
  \begin{thmE}
  Let $G$ be a finite group and $N$ a normal subgroup of $G$. Write $\rdim(G)=n$.  Then $$\rdim(G/N)\leq nj(n),$$
  where $j(n)$ is any bound in Jordan's theorem.
  \end{thmE}
  
 It is an old conjecture of Praeger and Easdown \cite{ep} that if $N\trianglelefteq G$ and $G/N$ is abelian, then $\mu(G/N)\leq\mu(G)$. This conjecture still remains open. In the case when $G$ is a $p$-group with an abelian maximal  subgroup it was proved in \cite{fra} that $\mu(G/G')\leq\mu(G)$. We will see that it is not true that $\rdim(G/G')\leq \rdim(G)$ even in this case. However, we can obtain the following bound.
 
 \begin{thmF}
   Let $G$ be a finite group and $N$ a normal subgroup of $G$ with $G/N$ abelian. Write $\rdim(G)=n$. Then
     $$\rdim(G/N)\leq3n/2.$$
\end{thmF}

This bound in Theorem F  depends on the classification of finite simple groups (this is the only result in the paper that relies on the CFSG). Without the CFSG, we can prove that $\rdim(G/N)\leq K n^2/\log n$ for some universal constant $K$.  Both versions of this result are straightforward consequences of known bounds on the number of generators of a linear group.
  
 As a consequence of Theorem E and \cite{rei}, which relies on a deep result in Mori theory \cite{bir}, we can obtain a new result on the essential dimension of an arbitrary finite group. It was asked in \cite{jly} whether  $\ed(G/N)\leq\ed(G)$ for any finite group $G$ and any $N\trianglelefteq G$. A negative answer to this question was given in Theorem 1.5 of \cite{mr}.  In fact, the example of  A. Meyer and Reichstein shows that we cannot hope for bounds better than exponential in Theorem E, even if we assume that $G$ is a $p$-group. We obtain the following bound for  $\ed(G/N)$ in terms of $\ed(G)$.
 
 \begin{corG}
 Let $G$ be a finite group and $N\trianglelefteq G$. Then $$\ed(G/N)\leq   \ed(G)h(\ed(G))j(\ed(G)h(\ed(G))),$$
 where $j$ is the bounding function in Jordan's theorem and $h$ is the bounding function in Birkar's Corollary 1.5 of \cite{bir}.
 \end{corG}
 
 We close this Introduction with a remark on the style that we have used in this paper. It is a paper on character theory of finite groups that, we hope, will be of interest to other areas outside group theory, particularly to those areas where the essential dimension of a finite group plays a role. For this reason, we have decided to include some details in our proofs that we would not have included in a paper addressed exclusively to group theorists. 
  
 \section{$p$-groups}
 
 Our approach will be character-theoretic.  Our notation follows \cite{isa}.   We start with the proofs of Theorems C and D. 
 If $\chi$ is a character of a finite group then $\chi$ can be decomposed as a sum of irreducible characters, called the irreducible constituents, and it is easy to see that the kernel of $\chi$, $\Ker\chi$, is the intersection of the kernels of the irreducible constituents (Lemma 2.21 of \cite{isa}). We thus have the first part of  the following elementary result.
 
 

 \begin{lem}
 \label{dec}
 Let $G$ be a finite group. Then 
 $$
 \rdim(G)=\min\{\sum_{i=1}^s\chi_i(1)\mid\textrm{$s\in\mathbb{Z}^+$, $\chi_i\in\Irr(G)$ for every $i=1,\dots,s$, $\bigcap_{i=1}^s\Ker\chi_i=1$}\}.
 $$
 Furthermore, if $\chi_i,\dots,\chi_s\in\Irr(G)$ are such that $\rdim(G)=\sum_{i=1}^s\chi_i(1)$ and $\bigcap_{i=1}^s\Ker\chi_i=1$  then for every $i=1,\dots,s$, $\Soc(G)\not\leq\Ker\chi_i$.
 \end{lem}
 
 \begin{proof}
 It suffices to prove the second part. This follows from the fact that if, say, $\Soc(G)\leq\Ker\chi_1$, then
 $$
 \Soc(G)\cap\left(\bigcap_{i=2}^r\Ker\chi_i\right)\leq\bigcap_{i=1}^r\Ker\chi_i=1
 $$
 so 
 $\bigcap_{i=2}^r\Ker\chi_i=1$ and $\sum_{i=1}^r\chi_i(1)<\rdim(G)$, contradicting the first part.
 \end{proof}
 
 Note that if $G$ is a finite group and $1<N\trianglelefteq G$, then $N\cap\Soc(G)>1$. Thus if $\chi$ is a (not necessarily irreducible) character of $G$, then $\chi$ is faithful if and only if $\Ker\chi\cap\Soc(G)=1$, which happens if and only if $\chi_{\Soc(G)}$ is faithful.
 
 If $G$ is a $p$-group, then the minimal normal subgroups have order $p$, so they are central. Thus $\Soc(G)=\Omega_1(Z(G))$, where $\Omega_1(Z(G))$ is the subgroup generated by the central elements of order $p$. This group is elementary abelian and its rank coincides with the rank of $Z(G)$. Recall also that if $A$ is a finite  abelian group then $\hat{A}=\Irr(A)$ is a group isomorphic to $A$. We have the following.
 
 \begin{lem}
 \label{ab}
 Let $p$ be a prime and let $A$ be an elementary abelian $p$-group.  Let $\mu$ be a character of $A$. Then $\mu$ is faithful if and only if the irreducible constituents of $\mu$  form a generating set of $\Irr(A)$.
 \end{lem}
 
 \begin{proof}
 Let $\lambda_1,\dots,\lambda_s$ be the irreducible constituents of $\mu$. We know that $\mu$ is faithful if and only if $\bigcap_{i=1}^s\Ker\lambda_i=1$. 
 Assume that these irreducible constituent do not form a generating set of $\Irr(A)$. Then $\Irr(A)=\langle\lambda_1,\dots,\lambda_s\rangle\times I$ for some $I>1$. By Problem 2.7 of \cite{isa}, there exists $1<B\leq A$ such that $B\leq\Ker\lambda_i$ for every $i=1,\dots,s$. Thus $\mu$ is not faithful.  
 The converse is proved analogously.
 \end{proof}

 If $G$ is any group, $N\trianglelefteq G$ and $\lambda\in\Irr(N)$, then we set
 $$
 \Irr(G|\lambda)=\{\chi\in\Irr(G)\mid [\chi_N,\lambda]\neq0\}.
 $$
 If $\chi$ is any of the characters in $\Irr(G|\lambda)$, we say that $\chi$ lies over $\lambda$.
Now we are ready to prove the following useful result to compute the representation dimension of $p$-groups.

 \begin{cor}
 \label{pgrp}
 Let $p$ be a prime, let $G>1$ be a group of order $p^n$ and let $r$ be the rank of $Z(G)$.  Let $\chi$ be a faithful character of minimal degree of $G$. Then there exist $\chi_1,\dots,\chi_r\in\Irr(G)$ such that $\chi=\chi_1+\cdots+\chi_r$ and $(\chi_i)_{\Omega_1(Z(G))}=e_i\lambda_i$ for some linear characters $\{\lambda_1,\dots,\lambda_r\}\subseteq\Irr(\Omega_1(Z(G)))$ that form a minimal generating set of $\Irr(\Omega_1(Z(G)))$. In particular, 
$$
\rdim(G)=\min_{\{\lambda_1,\dots,\lambda_r\}}\min\{\sum_{i=1}^r\chi_i(1)\mid\textrm{$\chi_i\in\Irr(G|\lambda_i)$ for every $i=1,\dots,r$}\},
$$
where $\{\lambda_1,\dots,\lambda_r\}$ runs over the minimal generating sets of $\Irr(\Omega_1(Z(G)))$.
Furthermore, $$\rdim(G)\leq rp^{(n-r)/2}.$$
\end{cor}

\begin{proof}
 Since $\chi$ is faithful, then $\mu=\chi_{\Omega_1(Z(G))}$ is faithful.  By Lemma \ref{ab}, the irreducible constituents of $\mu$ form a generating set $\{\lambda_1,\dots,\lambda_r\}$ of $\Irr(\Omega_1(Z(G)))$. 

We know that for any $i$, $\chi$ has some irreducible constituent $\chi_i$ lying over $\lambda_i$ and that for any choice of these irreducible constituents $\chi_i$, $\chi_1+\cdots+\chi_r$ is faithful. Since $\chi(1)$ is the minimal degree of a faithful character of $G$, we conclude that $\chi=\chi_1+\cdots+\chi_r$ with $\chi_i\in\Irr(G|\lambda_i)$ of minimal degree among the characters in $\Irr(G|\lambda_i)$ for every $i$. The result follows.
\end{proof}

The last statement in the previous lemma was also pointed out in the Introduction of \cite{ckr}.  




As promised in the Introduction to this paper,  we will deduce Theorems C and D from a more general result that characterizes the $p$-groups with center of rank $r$ with $\rdim(G)=rp^{\lfloor(n-r)/2\rfloor}$. Since the result has some differences according as to whether $n-r$ is even or odd we have split the result in two parts. We start with the $n-r$ even case.

 \begin{thm}
 \label{even}
 Let $p$ be a prime, let $G>1$ be a group of order $p^n$ and let $r$ be the rank of $Z(G)$. Assume that $n-r$ is even. Then $\rdim(G)=rp^{(n-r)/2}$ if and only if $Z(G)$ is elementary abelian of order $p^r$ and all characters in $\Irr(G|Z(G))$ are fully ramified with respect to $G/Z(G)$. 
 \end{thm}
 
 \begin{proof}
 First, note that by Corollary 2.30 of \cite{isa}, the degree of any irreducible character of $G$ is at most $|G:Z(G)|^{1/2}\leq p^{(n-r)/2}$. Furthermore, if is easy to deduce that if $G$ has irreducible characters of degree $p^{(n-r)/2}$,  then $Z(G)$ is elementary abelian.

 Let $G$ be a $p$-group of order $p^n$, with center elementary abelian of order $p^r$ and all characters in $\Irr(G|Z(G))$ fully ramified with respect to $G/Z(G)$.  We know by Corollary \ref{pgrp} that  $$\rdim(G)=\min_{\{\lambda_1,\dots,\lambda_r\}}\min\{\sum_{i=1}^r\chi_i(1)\mid\textrm{$\chi_i\in\Irr(G|\lambda_i)$ for every $i=1,\dots,r$}\},$$
 where $\{\lambda_1,\dots,\lambda_r\}$ runs over the minimal generating sets of $\Irr(\Omega_1(Z(G)))$. Fix a set $\{\lambda_1,\dots,\lambda_r\}$ that attains the first minimum. 
  By hypothesis, for any $\chi_i\in\Irr(G|\lambda_i)$, $\chi_i(1)=|G:Z(G)|^{1/2}=p^{(n-r)/2}$.  We deduce that $\rdim(G)=rp^{(n-r)/2}$, as desired. 
 
 Conversely, assume that  $\rdim(G)=rp^{(n-r)/2}$. Assume, by way of contradiction, that there exists $\chi_1\in\Irr(G|Z(G))$ with $\chi_1(1)<p^{(n-r)/2}$. Let $\lambda_1\in\Irr(\Omega_1(Z(G)))$ lying under $\chi_1$. Now, prolong  $\{\lambda_1\}$ to a minimal generating set $\{\lambda_1,\dots,\lambda_r\}$ of $\Irr(\Omega_1(Z(G)))$ and choose $\chi_i\in\Irr(G)$ lying over $\lambda_i$ for every $i=2,\dots,r$. Note that $\chi_i(1)\leq p^{(n-r)/2}$. Set $\chi=\chi_1+\cdots+\chi_r$ and note that $\chi(1)<rp^{(n-r)/2}$. Furthermore, all the members of $\{\lambda_1,\dots,\lambda_r\}$ are irreducible constituents of $\mu=\chi_{\Omega_1(Z(G))}$. By Lemma \ref{ab}, $\mu$ is faithful. Hence, $\chi$ is faithful. This contradicts the hypothesis  $\rdim(G)=rp^{(n-r)/2}$. Therefore, the degree of any character in $\Irr(G|Z(G))$ is $p^{(n-r)/2}$. The result follows.

  
  The next proof is similar and we omit some details.
   
 \begin{thm}
 \label{odd}
 Let $p$ be a prime, let $G>1$ be a group of order $p^n$ and let $r$ be the rank of $Z(G)$. Assume that $n-r$ is odd. Then $\rdim(G)=rp^{(n-r-1)/2}$ if and only if all characters in $\Irr(G|\Omega_1(Z(G)))$ have degree $p^{(n-r-1)/2}$.  In this case,   $|Z(G):\Omega_1(Z(G))|\leq p$, i.e.,  $Z(G)$ is either elementary abelian or isomorphic to $C_{p^2}\times C_p\times\cdots\times C_p$. 
 \end{thm}

  \begin{proof}
  Again, we note first that the degree of any irreducible character of $G$ is at most $p^{(n-r-1)/2}$. Furthermore,  if $G$ has irreducible characters of degree $p^{(n-r-1)/2}$,  then 
  $|Z(G):\Omega_1(Z(G))|\leq p$.
  
  If  all characters in $\Irr(G|\Omega_1(Z(G)))$ have degree $p^{(n-r-1)/2}$, then it follows from Corollary  \ref{pgrp}  that $\rdim(G)=rp^{(n-r-1)/2}$ (because a faithful character of minimal degree is the sum of $r$ characters in $\Irr(G|\Omega_1(Z(G)))$). 
  
  Conversely, assume that  $\rdim(G)=rp^{(n-r-1)/2}$. As in the previous theorem, we can see that the degree of any character in $\Irr(G|Z(G))$ is $p^{(n-r-1)/2}$. The result follows.
  \end{proof}

  \end{proof}
  
  Now, we are ready to deduce Theorem D. It turns out that we just need the case $n-r$ even.
  
  \begin{proof}[Proof of Theorem D]
  Let $r$ be the rank of $Z(G)$.
  Assume first that $n$ is odd. As mentioned in the table in p. 638 of \cite{ckr}, $f_p(n)=\max_{s\in\mathbb{N}}(sp^{\lfloor(n-s)/2\rfloor})=p^{(n-1)/2}$ in this case. 
It is easy to see that this maximum is achieved only at $s=1$.
Since  $f_p(n)=\rdim(G)\leq rp^{\lfloor(n-r)/2\rfloor}$ we deduce that $r=1$. 
Thus $n-r=n-1$ is even and we are in the situation of Theorem \ref{even}.  We deduce that $Z(G)$ has order $p$ and all characters in $\Irr(G|Z(G))$ are fully ramified with respect to $G/Z(G)$. 

Now, suppose that $n$ is even. In this case,  $f_p(n)=\max_{s\in\mathbb{N}}(sp^{\lfloor(n-s)/2\rfloor})=2p^{(n-2)/2}$ and it is easy to see that this maximum is achieved only at $s=2$.
As in the $n$ odd case, we can see that $r=2$. 
 Thus $n-r=n-2$ is even and we are also in the situation of Theorem \ref{even}.  We conclude that $Z(G)$ is elementary abelian of order $p^2$ and all characters in $\Irr(G|Z(G))$ are fully ramified with respect to $G/Z(G)$. 
 \end{proof}
 
 Since the proof of Theorem C is analogous to  that of Theorem D, we omit some details.
 
 \begin{proof}[Proof of Theorem C]
Let $r$ be the rank of $Z(G)$. If $n$ is odd, then $f_2(n)=3p^{(n-3)/2}$ and the maximum is achieved only at $s=3$. We can see that $r=3$, so $n-r$ is even and the result follows from Theorem \ref{even}.  If $n$ is even, then $f_2(n)=2p^{(n-2)/2}$ and the maximum is achieved only at $s=2$ and $s=4$. We can see that $r=2$ or $r=4$, so $n-r$ is even and the result also follows from Theorem \ref{even}. 
  \end{proof}
  
  
      

\section{Arbitrary groups} 

  In this section, we prove Theorems A and B. 
 If a group $G$ has a faithful irreducible character $\chi$, then it follows from Corollary 2.7 of \cite{isa} that $\chi(1)<\sqrt{|G|}$, so Theorem A holds in this case. The problem of which finite groups have faithful irreducible characters is therefore relevant for our purposes. This problem has been studied since the beginning of the 20th century and there are several, perhaps not very well-known, characterizations of these groups. We refer the reader to Section 2 of \cite{sze} for a nice description of the history of this problem. 
 
 As in the $p$-group case, the socle of $G$ is very relevant in these characterizations.  Recall that $\Soc(G)=A(G)\times T(G)$, where $A(G)=A_1\times\cdots\times A_t$ is a direct product of {\it some} elementary abelian minimal normal subgroups of $G$ and $T(G)$ is the direct product of {\it all} the nonabelian minimal normal subgroups of $G$ (see Definition 42.6 and Lemma 42.9 of \cite{hup}).  In the remaining of this article, we will use the notation introduced in this paragraph without further explicit mention. In particular, $t=t(G)$ is the number of elementary abelian minimal normal subgroups of $G$ that appear in a decomposition of $A(G)$ as a direct product of minimal normal subgroups.

 We will  use the following consequence of Gasch\"utz's characterization (Theorem 42.7 of \cite{hup}) of finite groups with a faithful irreducible character.
 
 \begin{thm}
 \label{gas}
 If for every prime $p$ every simple $\GF(p)G$-module appears at most with multiplicity one in $A(G)$, then $G$ has a faithful irreducible character. 
  \end{thm}
 
 \begin{proof}
 This is Theorem 42.12(a) of \cite{hup}. 
 \end{proof}

 \begin{lem}
 \label{t0}
 Let $G$ be a finite group. Then $T(G)$ has a faithful irreducible character. In particular, if $t=0$ (or, equivalently, if $G$ does not have any nontrivial abelian normal subgroup) then $G$ has a faithful irreducible character and $\rdim(G)<\sqrt{|G|}$. 
 \end{lem}
 
 \begin{proof}
 Note that $T(G)$ is a direct product of nonabelian simple groups. By Problem 4.3 of \cite{isa}, for instance, the product $\varphi$ of nonprincipal characters of each of the factors is a faithful irreducible character of $T(G)$. 
 
 If $\chi\in\Irr(G)$ lies over $\varphi$ then $\chi_{T(G)}$ is a sum of conjugates of $\varphi$, so by Lemma 2.21 of \cite{isa}, $\chi_{T(G)}$ is faithful. This implies that $$1=\Ker\chi\cap T(G)=\Ker\chi\cap\Soc(G),$$ so $\chi$ is a faithful irreducible character of $G$. Now, $\rdim(G)\leq\chi(1)<\sqrt{|G|}$, by Corollary 2.7 of \cite{isa}.
 \end{proof} 
 
 \begin{lem}
 \label{min}
 Let $G$ be a finite group without nonabelian minimal normal subgroups. Let $\chi$ be a faithful character of $G$ with $\rdim(G)=\chi(1)$. Then for every $\psi\in\Irr(G)$ irreducible constituent of $\chi$, $\Soc(G)\not\leq\Ker\psi$.
  \end{lem}
 
 \begin{proof}
 Assume not. Then $\Soc(G)=A(G)\leq\Ker\psi$ for some $\psi\in\Irr(G)$ irreducible constituent of $\chi$. Consider $\Delta=\chi-\psi$. Since $\psi$ is an irreducible constituent of $\chi$, $\Delta$ is a character of $G$ and $\Delta(1)<\chi(1)=\rdim(G)$. Thus $\Delta$ is not faithful. Let $\psi_1,\dots,\psi_s$ be the remaining irreducible constituents of $\chi$. Since $\Delta$ is not faithful, the intersection of the kernels of the $\psi_i$'s is not trivial. But since $\Ker\psi$ contains the whole socle, we deduce that the intersection of  the kernels of all the irreducible constituents of  $\chi$ is not trivial. This contradicts the hypothesis that $\chi$ is faithful.
 \end{proof}
 
 We write $d(G)$ to denote the rank of a group $G$. Recall that it is the minimal number of generators of $G$.
 The following result is well-known.

 \begin{lem}
 \label{abe}
 Let $G$ be a finite abelian group. Then $d(G)=\rdim(G)$.
 \end{lem}
 
 \begin{proof}
 Set $d(G)=m$.
 By the fundamental theorem of abelian groups, $G=C_1\times\cdots\times C_m$ is a direct product of $m$ cyclic groups $C_i=\langle x_i\rangle$. Let  $\lambda_i$ be a generator of $\Irr(C_i)$  and let $\mu_i\in\Irr(G)$ be the linear character determined by means of $\mu_i(x_i)=\varepsilon$, where $\varepsilon$ is an $o(x_i)$th primitive root of unity, and $\mu_i(x_j)=1$ for $j\neq i$. Notice that $\Ker\mu_i=C_1\cdots C_{i-1}C_{i+1}\cdots C_m$.  Put $\mu=\mu_1+\cdots+\mu_m$. By Lemma 2.21 of \cite{isa}, $\Ker\mu=1$, i.e., $\mu$ is faithful. Since $\mu(1)=m$ we deduce $\rdim(G)\leq\mu(1)=m=d(G)$. 
 
 Conversely, let $\chi$ be any faithful character of $G$. Decompose $\chi=a_1\chi_1+\cdots+a_s\chi_s$ as a sum of irreducible (linear)  characters $\chi_i$. Since $G/\Ker\chi_i$ is cyclic for every $i$ and the intersection of the kernels of the characters $\chi_i$ is trivial, we deduce that $G$ is isomorphic to a subgroup of the direct product of the cyclic groups $G/\Ker\chi_i$. Write $\Gamma$ to denote this group. Since $d(\Gamma)=s$ and $\Gamma$ is abelian, we deduce that $d(G)\leq d(\Gamma)=s\leq\chi(1)\leq \rdim(G)$. The result follows.
 \end{proof}

 The next result, in conjunction with Lemma \ref{dec}, lies at the core of our proof of Theorems A and B.

 \begin{lem}
 \label{key}
 Let $G$ be a finite group. Assume that $t>0$.  For $i=1,\dots, t$, write $B_i=A_1\times\cdots\times A_{i-1}\times A_{i+1}\times\cdots\times A_t$. Then 
 \begin{enumerate}
 \item
 For every $i=1,\dots,t$, there exists $\chi_i\in\Irr(G)$ such that $\Ker\chi_i\cap\Soc(G)=B_i$. Furthermore, 
 $$
 \chi_i(1)\leq|G/B_i:Z(G/B_i)|^{1/2}.
 $$
 \item
 We have $\bigcap_{i=1}^t\Ker\chi_i=1$. In particular, if $\chi=\chi_1+\cdots+\chi_t$, then $\rdim(G)\leq\chi(1)$. 
 \end{enumerate}
 \end{lem}
 
 \begin{proof}
  Let $\lambda_i\in\Irr(A_i)$ be nonprincipal for $i=1,\dots,t$ and $\varphi\in\Irr(T(G))$ be faithful. Recall that $A(G)=A_i\times B_i$.  Put
  $$\mu_i=\lambda_i\times1_{B_i}\times\varphi\in\Irr(Soc(G)).$$

    Let $\chi_i\in\Irr(G)$ lying over $\mu_i$. Since $\mu_i$ is an irreducible constituent of $(\chi_i)_{\Soc(G)}$,  $$\Ker\chi_i\cap \Soc(G)\leq\core_G(\Ker\mu_i)=B_i.$$ 
 Since $(\mu_i)_{B_i}$ is a multiple of the principal character, we clearly have that $B_i\leq\Ker\chi_i$. The first claim of part (i) follows. The second claim holds by Corollary 2.30 of \cite{isa}.
 
 By the definition of the subgroups $B_i$, their intersection is trivial. Thus 
 $$
 1=\bigcap_{i=1}^t(\Ker\chi_i\cap\Soc(G))=\left(\bigcap_{i=1}^t\Ker\chi_i\right)\cap\Soc(G).
 $$
 Since $\bigcap_{i=1}^t\Ker\chi_i$ is a normal subgroup of $G$, we deduce that it has to be the trivial subgroup. The inequality $\rdim(G)\leq\chi(1)$ follows from Lemma \ref{dec}.
 \end{proof}

      


    Now, we can obtain our first approximatiom to Theorem A when $t>0$.  In the remaining results in this section, we will also use the notation from Lemma \ref{key}. In particular, the characters $\chi_i$ and $\lambda_i$ will be the characters that have appeared in the statement of Lemma \ref{key} and its proof.

  \begin{lem}
  \label{gen2}
  Let $G$ be a finite group.  Assume that $t\geq1$. Write $|A_i|=a_i$ for every $i=1,\dots,t$. Then
  $$
  \rdim(G)<\sqrt{|G|}\left(\sum_{j=1}^t\prod_{k\neq j}\frac{1}{\sqrt{a_k}}\right).
  $$
  \end{lem}
  
  \begin{proof}
  Since $\chi_i\in\Irr(G/B_i)$,
we note that 
        $$\chi_i(1)<\sqrt{\frac{|G|}{a_1\cdots a_{i-1}a_{i+1}\cdots a_t}},$$
        (It suffices to observe that  $|B_i|=a_1\cdots a_{i-1}a_{i+1}\cdots a_t$.)

     Hence, if $\chi=\chi_1+\cdots+\chi_t$, 
     $$
    \rdim(G)\leq \chi(1)<\sum_{i=1}^t\sqrt\frac{|G|}{a_1\cdots a_{i-1}a_{i+1}\cdots a_t}=\sqrt{|G|}\left(\sum_{j=1}^t\prod_{k\neq j}\frac{1}{\sqrt{a_k}}\right),
     $$
     as desired
     \end{proof}
     
     Lemma \ref{gen2} implies that Theorem A holds when   $\sum_{j=1}^t\prod_{k\neq j}\frac{1}{\sqrt{a_k}}<1$. In the next elementary lemma we see that this is the case most of the times.
     
           \begin{lem}
     \label{cal}
     Let $t\geq2$ be an integer and let $a_1\geq\cdots\geq a_t\geq 2$ be $t$ integers. 
    If 
    $$
    \sum_{j=1}^t\prod_{k\neq j}\frac{1}{\sqrt{a_k}}\geq1
    $$ 
    then one of the following holds:
    \begin{enumerate}
    \item $t=2$ and $(a_1,a_2)\in\{(x,2), (y,3)\mid 2\leq x\leq11, 3\leq y\leq5\}$.
    \item $t=3$ and $(a_1,a_2,a_3)\in\{(x,2,2), (4,3,2), (3,3,2)\mid 2\leq x\leq 7\}$.
     \item $t=4$ and $(a_1,a_2,a_3,a_4)\in\{(x,2,2,2), (3,3,2,2)\mid 2\leq x\leq5\}$.
    \item $t=5$ and $(a_1,a_2,a_3,a_4,a_5)\in\{(2,2,2,2,2), (3,2,2,2,2)\}$.
    \end{enumerate}
    \end{lem}
    
    \begin{proof}
    Notice that in the expression     $\sum_{j=1}^t\prod_{k\neq j}\frac{1}{\sqrt{a_k}}$ we have $t$ summands and the denominator of each of the summands is at least $\sqrt{2^{t-1}}$. Therefore, each of the summands is at most $1/2^{(t-1)/2}$ so 
    $$
    \frac{t}{2^{(t-1)/2}}\geq \sum_{j=1}^t\prod_{k\neq j}\frac{1}{\sqrt{a_k}}>1
$$
and it follows from basic calculus that $t\leq 5$. The possible values for $(a_1,\cdots,a_t)$  for each of the possibilities for $t$ can also be obtained in an elementary way. We omit the details.
\end{proof}

Now, we can complete the proof of Theorems A and B by analyzing the exceptional cases that appear in Lemma \ref{cal}. We will use several times that if $G$ is a finite group and $\chi\in\Irr(G)$ then $\chi(1)\leq|G:Z(G)|^{1/2}$ (by Corollary 2.30 of \cite{isa}). Groups with an irreducible character $\chi$ such that $\chi(1)=|G:Z(G)|^{1/2}$ are called groups of central type. They have been rather studied. By a celebrated theorem of Howlett and Isaacs \cite{hi} they are solvable. We will not need the Howlett-Isaacs theorem, but we will use a  more elementary previous result that says that if $G$ is a group of central type then the set of primes that divide $|Z(G)|$ coincides with the set of primes that divide $|G|$ (see Theorem 2 of \cite{dj}). 

The next result includes both Theorem A and Theorem B. 
     
\begin{thm}
\label{sol}
  Let $G>1$ be a finite  group. Then one of the following holds:
  \begin{enumerate}
  \item
  $\rdim(G)<\sqrt{|G|}$; 
  \item
  $G$ is a $2$-group with socle $\Soc(G)=Z(G)=C_2\times C_2\times C_2$ and all characters in $\Irr(G|Z(G))$ are fully ramified with respect to $Z(G)$. For any such group $G$, $\rdim(G)=\frac{3}{\sqrt{8}}|G|^{1/2}$. 
  \item
   $G$ is a $2$-group with socle $\Soc(G)=Z(G)=C_2\times C_2$ and all characters in $\Irr(G|Z(G))$ are fully ramified with respect to $Z(G)$.  For any such group $G$, $\rdim(G)=\sqrt{|G|}$. 
   \item
      $G$ is a $2$-group with socle $\Soc(G)=Z(G)=C_2\times C_2\times C_2\times C_2$ and all characters in $\Irr(G|Z(G))$ are fully ramified with respect to $Z(G)$. For any such group $G$, $\rdim(G)=\sqrt{|G|}$.
   \end{enumerate}
\end{thm}

\begin{proof}
We have already seen that (i) holds if $t=0$. This also holds when 
$t=1$ by Theorem \ref{gas}.  Hence, using Lemmas \ref{gen2} and \ref{cal}, we may assume that $2\leq t\leq 5$. We consider these four cases separately.

{\bf Case 1:}
Assume first that $t=2$. We need to consider the values for $(a_1,a_2)$ that appear in Lemma \ref{cal}. By Theorem \ref{gas}, $G$ has an irreducible faithful character if $a_1\neq a_2$, so it suffices to consider the cases $(a_1,a_2)=(2,2)$ and  $(a_1,a_2)=(3,3)$.

{\bf Subcase 1.1:}
Suppose that $(a_1,a_2)=(2,2)$. Then $\chi_1\in\Irr(G/A_2)$ and $|A(G)/A_2|=2$, so $A(G)/A_2$ is central in $G$. Hence, $$\chi_1(1)\leq|G/A_2:Z(G/A_2)|^{1/2}\leq|G:A(G)|^{1/2}=|G|^{1/2}/2.$$ Arguing analogously with $\chi_2$, we obtain that $\chi_2(1)\leq |G|^{1/2}/2$, so $$\chi(1)=\chi_1(1)+\chi_2(1)\leq|G|^{1/2}.$$

If $\chi(1)=|G|^{1/2}$ then all inequalities so far are equalities. In particular, $A(G)=A_1\times A_2=Z(G)$ is a Klein 4-group  and $G$ is a group of central type. Since $Z(G)$ is a $2$-group, Theorem 2 of \cite{dj} implies that $G$ is also a $2$-group. Thus $T(G)=1$ and $\Soc(G)=A(G)=Z(G)$ is elementary abelian of order $4$.  Write $|G|=2^n$. By Corollary \ref{pgrp}, 
$$
\rdim(G)\leq2\cdot2^{(n-2)/2}=2^{n/2}=|G|^{1/2}.
$$
Furthermore, by Theorem D equality holds if and only if  $G$ is a $2$-group with socle $\Soc(G)=Z(G)=C_2\times C_2$ and all characters in $\Irr(G|Z(G))$ are fully ramified with respect to $Z(G)$.  We deduce that either (i) or (iii) holds.

     
 {\bf Subcase 1.2:} Now, we may assume that  $(a_1,a_2)=(3,3)$. Recall that $\chi_1$ is an irreducible character of $G$ that lies over a nonprincipal character of $A(G)/A_2$. Write $C/A_2=C_{G/A_2}(A(G)/A_2)$. Note that $G/C$ is isomorphic to a subgroup of $\Aut(A(G)/A_2)$ and since  $|A(G)/A_2|=3$, $|G/C|\leq2$. 
Notice also that     
   $A(G)/A_2$ is central in $C/A_2$. If $C=G$, then $A(G)/A_2$ is central in $G/A_2$ and $$\chi_1(1)\leq |G:A(G)|^{1/2}=|G|^{1/2}/3.$$  If $|G:C|=2$ and $\gamma\in\Irr(C)$ lies under $\chi_1$, then $\gamma(1)\leq|C:A(G)|^{1/2}$. By Clifford theory, $$\chi_1(1)\leq2\gamma(1)\leq2 |C:A(G)|^{1/2}=2\left(|G|/18\right)^{1/2}=|G|^{1/2}(2/3\sqrt{2}).$$
   Thus, in both cases, $\chi_1(1)\leq |G|^{1/2}(2/3\sqrt{2})$. Analogously,  $\chi_2(1)\leq |G|^{1/2}(2/3\sqrt{2})$. Hence,
   $$
   \chi(1)=\chi_1(1)+\chi_2(1)\leq |G|^{1/2}(4/3\sqrt{2})<|G|^{1/2},
   $$
   and (i) holds.

{\bf Case 2:} Now,  assume that $t=3$. Using Theorem \ref{gas} again,  together with lemmas \ref{gen2} and \ref{cal}, we may suppose that  $(a_1,a_2,a_3)=(x,2,2)$ for some $2\leq x\leq 7$ or $(a_1,a_2,a_3)=(3,3,2)$.
 
 {\bf Subcase 2.1:} 
 Suppose first that $(a_1,a_2,a_3)=(3,3,2)$.      Arguing as in Subcase 1.2, one can see that $\chi_1(1)\leq|G|^{1/2}/3,  \chi_2(1)\leq|G|^{1/2}/3$, and $\chi_3(1)\leq|G|^{1/2}/(3\sqrt{2})$. Thus 
 $$
 \chi(1)=\chi_1(1)+\chi_2(1)+\chi_3(1)\leq|G|^{1/2}(\frac{1}{3}+\frac{1}{3}+\frac{1}{3\sqrt{2}})<|G|^{1/2},
 $$
 and again (i) holds.

 {\bf Subcase 2.2:}
 Assume now that $(a_1,a_2,a_3)=(7,2,2)$.   Recall that $\chi_1$ is an irreducible character of $G$ that lies over a nonprincipal character of $A(G)/A_2A_3$. Write $C/A_2A_3=C_{G/A_2A_3}(A(G)/A_2A_3)$. Note that $G/C$ is isomorphic to a subgroup of $\Aut(A(G)/A_2A_3)$ and since  $|A(G)/A_2A_3|=7$, $|G/C|\leq6$. 
Arguing again as in previous cases, we have that  the worse bound for $\chi_1(1)$ is obtained when $|G:C|=6$ and in that case 
 $$\chi_1(1)\leq6 |C:A(G)|^{1/2}=6\left(|G|/168\right)^{1/2}=|G|^{1/2}(3/\sqrt{42}).$$ Also, $\chi_2(1)\leq|G|^{1/2}/2\sqrt{7}$ and $\chi_3(1)\leq |G|^{1/2}/2\sqrt{7}$.
 Thus 
 $$
 \chi(1)\leq|G|^{1/2}(\frac{3}{\sqrt{42}}+\frac{1}{2\sqrt{7}}+\frac{1}{2\sqrt{7}})<|G|^{1/2}.
 $$
 We conclude that (i) holds too.

The cases $(a_1,a_2,a_3)=(5,2,2)$ and $(a_1,a_2,a_3)=(3,2,2)$ are handled analogously.  We omit the details.


   {\bf Subcase 2.3:}
   Assume that $(a_1,a_2,a_3)=(4,2,2)$.  Recall that $\chi_2\in\Irr(G/A_1A_3)$. As before, $A(G)/A_1A_3$ is central in $G/A_1A_3$. Thus $\chi_2(1)\leq|G:A(G)|^{1/2}=|G|^{1/2}/4$. Assume first that $\chi_2(1)=|G:A(G)|^{1/2}$. Then $Z(G/A_1A_3)=A(G)/A_1A_3$ and $G/A_1A_3$ is a group of central type with $\chi_2$ fully ramified with respect to the center. Since $Z(G/A_1A_3)$ is a $2$-group and $G/A_1A_3$ is of central type, we deduce that $G/A_1A_3$, and hence $G$, is a $2$-group (by Theorem 2 of \cite{dj}).  Notice that $\chi_2(1)=|G|^{1/2}/4$. Analogously, we have that $\chi_3(1)\leq|G|^{1/2}/4$. Next, we bound $\chi_1(1)$. As usual, let $C/A_2A_3=C_{G/A_2A_3}(A(G)/A_2A_3)$. Recall that $G/C$ is isomorphic to a subgroup of $\Aut(A(G)/A_2A_3)=\Aut(C_2\times C_2)\cong S_3$ and since  $G$ is a $2$-group, $|G/C|\leq2$. As in previous cases, we deduce that 
   $$
   \chi_1(1)\leq2|C:A(G)|^{1/2}=|G|^{1/2}(1/2\sqrt{2}).
   $$
   We conclude that $\chi(1)<|G|^{1/2}$. Hence, we may assume that  $\chi_2(1)<|G:A(G)|^{1/2}$. This implies that there exist at least two irreducible characters of $G$ lying over the 
nonprincipal irreducible character $\lambda_2$ of $A(G)/A_1A_3$. Since 
$$
\sum_{\chi\in\Irr(G/A_1A_3|\lambda_2)}\chi(1)^2=|G:A(G)|,
$$
we deduce that for some of the characters $\chi$ in this sum, $\chi(1)^2\leq|G:A(G)|/2=|G|/32$. Hence, we may assume that $\chi_2(1)\leq|G|^{1/2}/4\sqrt{2}$. Repeating the same reasoning, we may also assume that  $\chi_3(1)\leq|G|^{1/2}/4\sqrt{2}$. Now, we bound $\chi_1(1)$. With our usual notation and arguments, we may assume that $|G:C|=6$ and one can see that 
$$
\chi_1(1)\leq6|C:A(G)|^{1/2}=6\left(|G|^{1/2}/4\sqrt{6}\right)=|G|^{1/2}\sqrt{6}/4.
$$
Thus 
$$\chi(1)\leq|G|^{1/2}(\frac{\sqrt{6}}{4}+\frac{1}{4\sqrt{2}}+\frac{1}{4\sqrt{2}})<|G|^{1/2}.
$$
In this case, (i) also holds.

{\bf Subcase 2.4:}
Finally, we may assume that  $(a_1,a_2,a_3)=(2,2,2)$. In particular, $A(G)\leq Z(G)$. If $A(G)<Z(G)$ then $|Z(G)|\geq2^4$ and $\psi(1)\leq|G:Z(G)|^{1/2}\leq|G|^{1/2}/4$ for every $\psi\in\Irr(G)$. Since $\chi$ is the sum of $3$ irreducible characters of $G$, we deduce that $\chi(1)<|G|^{1/2}$ and (i) holds. Thus we may assume that $A(G)=Z(G)$.  In particular, if $\psi\in\Irr(G)$
then $\psi(1)\leq|G:Z(G)|^{1/2}\leq|G|^{1/2}/\sqrt{8}$. 

Assume that $\rdim(G)\geq\sqrt{|G|}$. 
Let $\lambda\in\Irr(Z(G))$ be a nonprincipal character and $K=\Ker\lambda$. We claim that $Z(G/K)=Z(G)/K$.    We argue by way of contradiction. Assume that $Z(G/K)>Z(G)/K$.
 Let $\psi\in\Irr(G|\lambda)$. Then $\psi(1)\leq|G/K:Z(G/K)|^{1/2}\leq|G|^{1/2}/4$.    We deduce that there exists a faithful character $\chi$ of $G$ such that 
 $$
 \chi(1)\leq|G|^{1/2}/4+|G|^{1/2}/\sqrt{8}+|G|^{1/2}/\sqrt{8}<|G|^{1/2}.
 $$
 This is a contradiction. Hence, we have proved the claim. This argument also shows that $\psi(1)=|G:Z(G)|^{1/2}=|G|^{1/2}/\sqrt{8}$ for every $\psi\in\Irr(G|Z(G))$, as desired. We deduce that (ii) holds.
 
 The remaining two cases can be handled with the same techniques. Therefore, we will omit most details in Cases 3 and 4.
 
 {\bf Case 3:}   Assume that $t=4$. Again, using Theorem \ref{gas},  together with Lemma \ref{gen2} and \ref{cal}, we may suppose that  $(a_1,a_2,a_3,a_4)=(2,2,2,2)$ or $(a_1,a_2,a_3,a_4)=(3,3,2,2)$. In the first subcase, one can see arguing as in Subcase 1.1 that either (i) or (iv) holds. In the second subcase, it follows from an analysis of the degrees of $\chi_i$, $i=1,\dots,4$, that $\rdim(G)<\sqrt{|G|}$.  

{\bf Case 4:} Finally, assume that $t=5$. Again, using Theorem \ref{gas},  together with Lemma \ref{gen2} and \ref{cal}, we may suppose that  $(a_1,a_2,a_3,a_4,a_5)=(2,2,2,2,2)$.  
Thus $A(G)$ is the direct product of $5$ minimal normal subgroups of order $2$. Thus $A(G)$ is central. Hence $\chi_i(1)\leq|G|^{1/2}/\sqrt{32}$ for $i=1,\dots,5$ and $\chi(1)<|G|^{1/2}$. We deduce that (i) holds.

Now, it remains to determine $\rdim(G)$ when $G$ is one of the groups that appear in part (ii), (iii) or (iv). This has been done in Theorem \ref{even}.
  \end{proof}

  

  



        
    \section{Proof of Theorem E, Theorem F and Corollary G}
    
    In this section we provide the short proofs of the remaining results. 
    We start with  Theorem E.
    
    \begin{proof}[Proof of Theorem E]
    By Jordan's theorem, there exists and abelian subgroup $A\leq G$ such that $|G:A|\leq j(n)$ for some function $j$. Since $NA/N\cong A/A\cap N$ is abelian, we deduce that 
    $$
    \rdim(NA/N)=d(NA/N)\leq d(A)\leq n.
    $$
    Thus $NA/N$ has a faithful character $\Delta$ of degree at most $n$. Hence the induced character $\Delta^G$ is faithful and has degree at most $nj(n)$. The result follows.
    \end{proof}
    
    Now, we deduce Corollary G.
    
    \begin{proof}[Proof of Corollary G]
    Write $\rdim(G)=n$.
  By Proposition 4.15 of \cite{bf}, 
 $$\ed(G/N)\leq \rdim(G/N).$$ Furthermore, by Theorem E, 
  $\rdim(G/N)\leq nj(n)$.
  On the other hand, by Theorem 2 of  \cite{rei}, $n\leq \ed(G)h(\ed(G))$. The result follows.
  \end{proof}

As we already mentioned in the Introduction, it is not true that $\rdim(G/N)\leq \rdim(G)$ when $G/N$ does not have any nontrivial abelian normal subgroup. This is false even when $G$ is a $p$-groups with an abelian maximal subgroup: consider $G={\tt SmallGroup(2^5,38)}$. This group has an abelian maximal subgroup, it has faithful irreducible characters of degree $2$, but $\rdim(G/G')=d(G/G')=3$. Theorem F follows immediately from a result of Robinson and K\'ovacs.
    
\begin{proof}[Proof of Theorem F]
Since $G/N$ is abelian, $\rdim(G/N)=d(G/N)\leq d(G)$. Now, the result follows from \cite{kr}.
\end{proof}

Note that the bound in \cite{kr} relies on the classification. Using the weaker (but classification-free) bound in \cite{fis}, we get that $\rdim(G/N)\leq Kn^2/\log n$ for some constant $K$.

   \section{Concluding remarks and questions}
   
   The Heisenberg groups mentioned in the Introduction are just one example of the  nonabelian groups that   appear in the statement of Theorems A and B. Any semiextraspecial group with center of the specified order also satisfies those hypotheses. As discussed in \cite{lew}, these form a rather large family of groups. However, all of them have class $2$. This suggests the question of whether or not there are groups of nilpotence class larger than $2$ among the exceptional groups in Theorem A and Theorem B. Using GAP \cite{gap} we have found groups of order $2^8$ and  nilpotence class $3$ with the properties of those in the statement of Theorem B(i) (for instance, ${\tt SmallGroup}(2^8, 3196)$). We suspect that there should also exist $2$-groups of class $3$ among the exceptional groups in Theorem A and also among those in Theorem B(ii). In fact, we expect the following question to have an affirmative answer:
   
   \begin{que}
   \label{1}
   Let $r\geq1$ be an integer. Do there exist $p$-groups $G$ with $Z(G)$ elementary abelian of order $p^r$ and all characters in $\Irr(G|Z(G))$ fully ramified with respect to $G/Z(G)$ of arbitrarily large nilpotence class? 
      \end{que}
   
   We have been informed by Z. Reichstein that he asked this question in the cases $r=1$ and $r=2$ at a conference in Banff on permutation groups in 2009.  More precisely, he asked whether given a prime $p$ and a positive integer $n$, there exists a $p$-group $G$ of order $p^n$, with maximal representation dimension among groups of order $p^n$, and nilpoitence class $>2$. Subsequently, C. Parker and R. Wilson constructed groups $G$  of order $p^{2p+3}$, for any odd prime $p$. This appears in \cite{cer}. These are examples of groups of class $>2$ that satisfy the conditions of Question \ref{1} with $r=1$. It turns out that Question \ref{1} has an affirmative answer when $r=1$ 
    by Theorem 6.3 of \cite{gag}. 

We remark also that the condition that all characters in $\Irr(G|Z(G))$ are fully ramified with respect to $G/Z(G)$ is equivalent to the condition that $(G,Z(G))$ is a Camina pair. We refer the reader to \cite{lew12}, where these groups were studied (see also \cite{lew21}). In particular, there is a character-free characterization of these groups (in terms of conjugacy classes). This seems tougher for the groups that appear in Theorem \ref{odd}

\begin{que}
\label{2}
Describe the $p$-groups $G$ that are not of central type with all characters in $\Irr(G|\Omega_1(Z(G)))$ of the same degree.  Is there a character-free characterization of these groups? Are there groups of arbitrarily large nilpotence class among them?
\end{que}

By the above-mentioned theorem of S.  Gagola, any $p$-group is isomorphic to a subgroup of $G/Z(G)$ for some $p$-group $G$ of central type with $Z(G)$ of order $p$, so it would be interesting to decide whether the same happens when the degree of the common degree of the characters in $\Irr(G|\Omega_1(Z(G)))$ is not $|G:Z(G)|^{1/2}$.

   The proof of Theorem 1.5 of \cite{mr} shows that the best bounds one can hope for in Theorem E and Corollary G are exponential. The bounds we have obtained, even with the help of  results that depend on the classification of finite simple groups, are super-exponential. This, together with the known bounds and examples known for the analog problem for permutation representations mentioned in the Introduction,  suggests the following questions. (Note, however, that we have already seen several differences between results for minimal faithful permutation representations and minimal dimensions of faithful linear representations.)
   
   \begin{que}
Does there exist a constant $c_1>1$ such that if $G$ is a finite group and $N\trianglelefteq G$ then  $\rdim(G/N)\leq c_1^{\rdim(G)}$?
   \end{que}
   
   \begin{que}
Does there exist a constant $c_2>1$ such that if $G$ is a finite group and $N\trianglelefteq G$ then  $\ed(G/N)\leq c_2^{\ed(G)}$?
   \end{que}


\begin{thebibliography}{131}
  
  \bibitem{bms1} 
{\sc M. Bardestani, K. Mallahi-Karai, H. Salmasian},  Minimal dimension of faithful representations for $p$-groups. \emph{J. Group Theory} {\bf 19} (2016), 589--608.

  \bibitem{bms2} 
{\sc M. Bardestani, K. Mallahi-Karai, H. Salmasian},  Kirillov's orbit method and polynomiality of the faithful dimension of $p$-groups. \emph{Compos. Math.} {\bf 155} (2019), 1618--1654.


    \bibitem{bf} 
{\sc G. Berhuy, G. Favi},  Essential dimension: a functorial point of view (after Merkurjev). \emph{Doc. Math.} {\bf 8} (2003), 279--330.

\bibitem{bir}
{\sc C. Birkar}, Singularities of linear systems and boundedness of Fano varieties, {\emph Ann. of Math.} {\bf 193}  (2021), 347--405.
  
    \bibitem{bg} 
{\sc J. Bourgain, A. Gamburd},  Uniform expansion bounds for Cayley graphs of $\SL_2(\mathbb{F}_p)$. \emph{Ann. of Math.} {\bf 167} (2008), 625--642.


 \bibitem{br} 
{\sc J. Buhler, Z. Reichstein},  On the essential dimension of a finite group. \emph{Compos. Math.} {\bf 106} (1997), 159--179.

\bibitem{lew21}
{\sc S. Burkett, M. Lewis}, Partial GVZ-groups.  \emph{J. Algebra} {\bf 581} (2021), 50--62.

\bibitem{cer}
{\sc S. Cernele}, Maximal representation dimension for groups of order $p^n$. Master's Essay, University of British Columbia, 2010.


 \bibitem{ckr} 
{\sc S. Cernele, M. Kamgarpour, Z. Reichstein},  Maximal representation dimension of finite $p$-groups. \emph{J. Group Theory} {\bf 14} (2011), 637--647.


 \bibitem{col} 
{\sc M. Collins},  On Jordan's theorem for complex linear groups. \emph{J. Group Theory} {\bf 10} (2007), 411--423.


\bibitem{dj}
{\sc  F. R. DeMeyer, G. J. Janusz},  Finite groups with an irreducible representation of large degree. \emph{Math. Z.} {\bf 108} (1969), 145--153.



    \bibitem{ep} 
{\sc D. Easdown, C. Praeger},  On minimal faithful permutation representations of finite groups. \emph{Bull. Austral. Math. Soc.} {\bf 38} (1988), 207--220.





    \bibitem{fis} 
{\sc R. K. Fisher},  The number of generators of finite linear groups. \emph{Bull. London Math. Soc.} {\bf 6} (1974), 10--12.



    \bibitem{fra} 
{\sc C. Franchi},  On minimal degrees of  permutation representations of abelian quotients of finite groups. \emph{Bull. Austral. Math. Soc.} {\bf 84} (2011), 408--413.

 
 \bibitem{gag}
 {\sc S. Gagola, Jr.}, Characters vanishing on all but two conjugacy classes \emph{Pacific J.  Math. } {\bf 109} (1983), 363--385.
 
\bibitem{gap}
{\sc The GAP group}, \emph{{\sf GAP} - Groups, Algorithms, and Programming}.
  Version 4.4, 2004, {\sf http://www.gap-system.org}.


   \bibitem{gow} 
{\sc W. Gowers},  Quasirandom groups. \emph{Combin. Probab. Comput.} {\bf 17} (2008), 363--387.




    \bibitem{hw} 
{\sc D. Holt, J. Walton},  Representing the quotient groups of a finite permutation group. \emph{J. Algebra} {\bf 248} (2002), 307--333.

  
  \bibitem{hi}
  {\sc R. B. Howlett, I. M. Isaacs}, On groups of central type. \emph{Math. Z.}  {\bf 179} (1982), 555--569.


  

\bibitem{hup}
{\sc B. Huppert}, \emph{Character Theory of Finite Groups}. de Gruyter, 
Berlin, 1998.

  
  

  
  
  
 \bibitem{isa}
{\sc I. M. Isaacs}, \emph{Character Theory of Finite Groups}. AMS-Chelsea,
  Providence, 2006.
  
  
    \bibitem{jly} 
{\sc C. Jensen, A. Ledet, N. Yui},  Generic polynomials. \emph{Math. Sci. Res. Inst. Publ} {\bf 45} Cambridge University Press, Cambridge 2002.

  
  
\bibitem{km} 
{\sc N. Karpenko, A. Merkurjev},  Essential dimension of finite $p$-groups. \emph{Invent. Math.} {\bf 172} (2008), 491--508.


    \bibitem{kp} 
{\sc L. Kov\'acs, C. Praeger},  On minimal faithful permutation representations of finite groups. \emph{Bull. Austral. Math. Soc.} {\bf 62} (2000), 311--317.

 \bibitem{kr} 
{\sc L. Kov\'acs, G. Robinson},  Generating finite completely reducible linear groups. \emph{Proc. Amer. Math. Soc.} {\bf 112} (1991), 357--364.


\bibitem{lew12}
{\sc M. Lewis}, On $p$-group Camina pairs.  \emph{J. Group Theory} {\bf 15} (2012), 469--483.

\bibitem{lew}
{\sc M. Lewis},  Semi-extraspecial groups. In: Advances in algebra, 219-237, Springer Proc. Math. Stat. {\bf 277} Springer, 2019.





\bibitem{lmm}
{\sc A. Lucchini, F. Menegazzo, M. Morigi},  On the number of generators and composition length of finite linear groups. \emph{J. Algebra} {\bf 243} (2001), 427--447.

  
  

\bibitem{mer1} 
{\sc A. Merkurjev},  Essential dimension: a survey \emph{Transform. Groups} {\bf 18} (2013), 415--481.


 \bibitem{mer2} 
{\sc A. Merkurjev},  Essential dimension. \emph{Bull. Amer. Math. Soc.} {\bf 54} (2017), 635--661.


 \bibitem{mr} 
{\sc A. Meyer, Z. Reichstein},  Some consequences of the Kerpenko-Merkurjev theorem. \emph{Doc. Math.}  (2010), 445--457.


 \bibitem{neu} 
{\sc P. Neumann},  Some algorithms for computing with finite permutation groups.  In: Groups-St. Andrews 1985, LMS Lecture Notes Series {\bf 121}, Cambridge University Press, 1987, 59--92.

 \bibitem{rei} 
{\sc Z. Reichstein},  The Jordan property of Cremona groups and essential dimension. \emph{Arch. Math.} {\bf 111} (2018), 449--455.

  \bibitem{rob} 
{\sc G. Robinson},  On linear groups. \emph{J. Algebra} {\bf 131} (1990), 527--534.


\bibitem{sze}
{\sc F. Szechtman},  Groups having a faithful irreducible representation. \emph{J. Algebra} {\bf 454} (2016), 292--307.


  \bibitem{tot} 
{\sc B. Totaro}, \emph{Group Cohomology and Algebraic Cycles}. Cambridge Tracts in Mathematics  Cambridge, 2014.

\end{thebibliography}
\end{document}